\documentclass[11pt]{amsart}
\usepackage{amsthm, amsfonts, amssymb, amsmath, amscd, mathtools}
\usepackage{mathrsfs}
\usepackage[all]{xy}
\usepackage{lmodern}
\usepackage[margin=3cm]{geometry}
\usepackage{hyperref}
\usepackage{xcolor}
\usepackage{comment}

\newcommand{\into}{\hookrightarrow}

\newcommand{\C}{\mathbb{C}}
\newcommand{\R}{\mathbb{R}}

\newcommand{\Z}{\mathbb{Z}}
\newcommand{\N}{\mathbb{N}}

\newcommand{\G}{\Gamma}

\newcommand{\T}{\mathbb{T}}

\newcommand{\Cl}{\mathbb{C}{\rm l}}

\newcommand{\aHom}{{\rm Hom}_{\rm alg}}
\newcommand{\rees}{{\rm Rees}}

\newcommand{\corr}[1]{\overset{#1}{\longrightarrow}}

\newcommand{\la}{\langle}
\newcommand{\ra}{\rangle}
\newcommand{\bd}{\partial}

\newcommand{\sset}{\subseteq}
\newcommand*{\van}[1]{\mathcal{I}_{#1}}

\newcommand{\nuw}{\nu_\mathcal{W}}
\newcommand{\defw}{\delta_\mathcal{W}}
\newcommand{\gr}{{\rm gr}}

\newcommand{\DO}{{\rm DO}}
\newcommand{\WT}[1]{\widetilde{#1}}

\numberwithin{equation}{section}

\usepackage{microtype}
\numberwithin{equation}{section}
\theoremstyle{theorem}
\newtheorem{theorem}[equation]{Theorem}
\newtheorem{lemma}[equation]{Lemma}
\newtheorem{proposition}[equation]{Proposition}
\newtheorem{corollary}[equation]{Corollary}

\theoremstyle{definition}
\newtheorem{definition}[equation]{Definition}

\theoremstyle{remark}
\newtheorem{remark}[equation]{Remark}
\newtheorem{remarks}[equation]{Remarks}
\newtheorem{example}[equation]{Example}

\begin{document}
	\title[Linear weightings]{Linear weightings and deformations of vector bundles}

	\author{Daniel Hudson}
	\email{dhudson@math.toronto.edu}
	\address{Department of Mathematics\\
	University of Toronto\\
	40 St. George Street\\
	Toronto, ON\\
	Canada M5S 2E4}

	
	\date{\today}
	
	
	\maketitle
	
	\begin{abstract}
            We define and explore the notion of \emph{linear weightings} for vector bundles, extending the recent work by Loizides and Meinrenken in~\cite{LOIZ-MEIN2020}. We construct weighted normal bundles and deformation spaces in the category of vector bundles. We explain how a linear weighting determines a filtration on differential operators, and an interpolation between a differential operator and its ``weighted" linearization. We close by explaining how our constructions capture the rescaled spinor bundle of Higson and Yi, constructed in~\cite{higson2019spinors}, and a related construction by \v{S}evera in~\cite{vsevera2017letters}. 
	\end{abstract}

 \section{Introduction}

The purpose of this paper is to provide a new perspective on the rescaled spinor bundle of Higson and Yi ~\cite{higson2019spinors}, from the point of view of weightings recently introduced by Loizides and Meinrenken. We accomplish this by introducing the notion of \emph{linear weightings} of vector bundles, modelled on the definition of a weighting given by Loizides and Meinrenken in~\cite{LOIZ-MEIN2020}.   

Let us briefly recall the rescaled spinor bundle of Higson and Yi. Suppose that $M$ is a Riemannian spin manifold with spinor bundle $S\to M$. In order to understand Getzler's approach (\cite{berline2003heat}, ~\cite{getzler1983pseudodifferential}) to the index theorem from the perspective Connes' tangent groupoid (\cite{connes1994}), Higson and Yi introduce the \emph{rescaled spinor bundle} $\mathbb{S}$, which is a vector bundle over the tangent groupoid $\T M = TM \sqcup (M\times M\times \R^\times)$. It can be understood as a family of vector bundles parameterized by $\R$, given by 
    \begin{equation*}
        \mathbb{S}_t = \left\{
            \begin{array}{ll}
                S\boxtimes S^* & t\neq 0  \\
                \pi^*(\wedge^\bullet TM) & t = 0, 
            \end{array}
        \right.
    \end{equation*}
where $\pi:TM \to M$ is the vector bundle projection and $S\boxtimes S^*$ is the vector bundle over $M\times M$ with fibre $(S\boxtimes S^*)_{(m_1,m_2)} = S_{m_1}\otimes S^*_{m_2}$. 

The rescaled spinor bundle can be thought as a deformation space. Deformation spaces were studied in depth by Loizides and Meinrenken in their recent work on \emph{weightings} (\cite{LOIZ-MEIN2020}). Roughly speaking, a weighting of a manifold $M$ is given by assigning weights to local coordinates. This manifests as a multiplicative filtration 
    \begin{equation}
    \label{IntroEquation: weighting}
        C^\infty_M = C^\infty_{M,(0)} \supseteq C^\infty_{M,(1)} \supseteq C^\infty_{M,(2)} \supseteq \cdots 
    \end{equation}
of the sheaf of smooth functions on $M$. Loizides and Meinrenken show that $C^\infty_{M,(1)}$ is always the vanishing ideal of some closed submanifold $N\sset M$, and the weighting of $M$ determines a fibre bundle over $N$,
    \[ \nuw(M,N) \to N, \]
called the \emph{weighted normal bundle}. Loizides and Meinrenken also define a \emph{weighted deformation space} $\defw(M,N)$, admitting a set-theoretic decomposition 
    \[ \defw(M,N) = \nuw(M,N) \sqcup (M\times \R^\times); \]
see section~\ref{subsection: preliminaries} for a brief review of their constructions. For example, if the filtration~\eqref{IntroEquation: weighting} is given by order of vanishing along the diagonal $M \into M\times M$, then $\defw(M\times M, M)$ is the tangent groupoid $\T M$. 

We unify the constructions of the rescaled spinor bundle $\mathbb{S}$ and the weighted deformation space $\defw(M,N)$ by introducing \emph{linear weightings}. In addition to providing conceptual clarity, this gives a more systematic approach to constructing deformations of vector bundles. It also allows for one to take into account more exotic geometry on the base, for instance Carnot manifolds (\cite{van2017tangent}) or manifolds with boundary (\cite{melrose1993atiyah}).

Let us outline this article. We begin by defining a linear weighting of a vector bundle $V\to M$. This is given by a weighting of $M$ and a $\Z$-graded module filtration 
    \begin{equation}
    \label{IntroEquation: linear weighting}
        \cdots \supseteq \G(V)_{(i)} \supseteq \G(V)_{(i+1)} \supseteq \cdots 
    \end{equation}
of the sheaf of sections of $V$, satisfying some local assumptions.  We explain how a linear weighting of $V$ determines linear weightings of $V^*$, $\wedge^k V$, ${\rm Sym}^k V$, et cetera. We then show that a linear weighting of $V\to M$ is equivalent to a multiplicative filtration 
    \[ \cdots \supseteq C^\infty_{pol}(V)_{(i)} \supseteq C^\infty_{pol}(V)_{(i+1)} \supseteq \cdots \]
of the sheaf of polynomial functions on $V$. We explain how Loizides and Meinrenken's constructions of $\nuw(M,N)$ and $\defw(M,N)$ generalize to linear weightings, giving vector bundles 
    \[ \nuw(V) \to \nuw(M,N) \quad \text{and} \quad \defw(V) \to \defw(M,N)  \]
such that $\defw(V)$ can be understood as a family of vector bundles 
    \begin{equation*}
     \xymatrix{
        \defw(V) = \nuw(V) \sqcup (V \times \R^\times) \ar[d] \\
        \defw(M,N) = \nuw(M,N) \sqcup (M\times \R^\times).
    }
    \end{equation*} 
We explain in Section~\ref{section: examples in the literature} how our construction captures the rescaled spinor bundle of Higson and Yi, as well as a related construction by \v{S}evera in his letters to Weinstein (\cite{vsevera2017letters}). 

We expect that our work, combined with the characterization of pseudo-differential operators recently obtained by van Erp and Yuncken in~\cite{van2019groupoid}, can be used to define interesting calculi of operators acting on sections of vector bundles. We intend to pursue this direction in future work. 

\subsection*{Acknowledgements}

The author wishes to thank Gabriel Beiner, Yiannis Loizides, and Eckhard Meinrenken for their helpful discussions and comments, especially regarding \v{S}evera's algebroid. The author would also like to thank the University of Victoria, and especially Jane Butterfield, for giving him an office to work in during the COVID-19 pandemic. This research was supported by an NSERC CGS-D. 

\section{Linear weightings}
\label{section: linear weightings}

\subsection{Preliminaries}
\label{subsection: preliminaries}

We being by fixing our conventions for vector bundles and by giving a brief introduction to the theory of weightings. Throughout this article $M$ always denotes a $C^\infty$-manifold, unless explicitly said otherwise. 

\subsubsection{Vector bundles}

In this note all vector bundles $V\to M$ will be smooth and finite rank. We let $\G(V)$ denote the smooth sections of $V$. We write $C^\infty_{[n]}(V)$ to denote the space of smooth functions on $V$ which are homogeneous of degree $n$. There are canonical isomorphisms 
    \[ C^\infty_{[0]}(V) \cong C^\infty(M) \quad \text{and} \quad C^\infty_{[1]}(V) = \G(V^*) \]
which we make use of without comment. The space of polynomials on $V$ is the direct sum 
    \[ C^\infty_{pol}(V) = \bigoplus_{n\geq 0} C^\infty_{[n]}(V). \]
By \emph{vector bundle coordinates} we mean a coordinate system $x_a, y_b$ defined on $V|_U$ for some open set $U\sset M$, where $x_a \in C^\infty_{[0]}(V|_U)$ and $y_b \in C^\infty_{[1]}(V|_U)$.

By a \emph{subbundle} of $V$ we mean a submanifold $W\sset V$ with the property that for each $w\in W\cap V|_U$ there exist vector bundle coordinates $x_a \in C^\infty_{[0]}(V|_U)$ and $y_b \in C^\infty_{[1]}(V|_U)$ so that 
    \[ W\cap V|_U =  \{x_1 = \cdots = x_n = y_1 = \cdots = y_r = 0\}. \]
Equivalently, we may define subbundles as submanifolds of $V$ which are invariant under scalar multiplication (\cite{grabowski2009higher}). In particular, a subbundle $W\sset V$ is itself a vector bundle whose base is a (potentially proper) submanifold of $M$; in the case when the base of $W$ is all of $M$, we refer to $W$ as a \emph{wide} subbundle. If $W\to N$ is a subbundle of $V\to M$, then we write 
    \[ \G(V, W) = \{ \sigma \in \G(V) : \sigma|_N \in W \}.  \]

\subsubsection{Review of weightings}

Let $w = (w_1, \dots, w_m)\in \Z^m_{\geq 0}$.  A weighting of $M$ is a multiplicative filtration 
    \begin{equation}
    \label{equation: LM defn of weighting}
        C^\infty_M = C^\infty_{M,(0)} \supseteq C^\infty_{M,(1)} \supseteq C^\infty_{M,(2)} \supseteq \cdots
    \end{equation}
of the sheaf of smooth functions on $M$ with the property that locally one can find coordinates $x_1, \dots, x_m \in C^\infty(U)$ so that $C^\infty(U)_{(i)}$ is the ideal generated by the monomials 
    \[ x^\alpha = x_1^{\alpha_1} \cdots x_m^{\alpha_m} \quad \text{with } \alpha \cdot w = \sum_{j}\alpha_jw_j \geq i;\]
we refer to $w$ as the \emph{weight vector}, or \emph{weights}, for $M$, and the local coordinates $x_a$ as \emph{weighted coordinates}. By~\cite[Lemma 2.4]{LOIZ-MEIN2020}, $C^\infty_{M,(1)}$ is the vanishing ideal of a closed submanifold $N\sset M$; if $N$ is given in advance we say that $M$ \emph{is weighted along $N$} and refer to $(M,N)$ as a \emph{weighted pair}. A weighted morphism $\varphi:(M,N)\to (M',N')$ is a smooth map whose pullback is filtration preserving. 

Since we are working in the $C^\infty$-category, we may take advantage of the existence of partitions of unity to avoid the use of sheaves, working instead with the filtration of global functions 
        \[ C^\infty(M) = C^\infty(M)_{(0)} \supseteq C^\infty(M)_{(1)} \supseteq C^\infty(M)_{(2)} \supseteq \cdots . \]

Motivated by the algebro-geometric description of the tangent groupoid due to Haj and Higson (\cite{sadegh2018euler}), Loizides and Meinrenken associate a \emph{weighted normal bundle} and \emph{weighted deformation space} to a weighted manifold $M$. The weighted normal bundle associated to the weighted pair $(M,N)$ is the character spectrum 
    \[ \nuw(M,N) = \aHom(\gr(C^\infty(M)), \R),  \]
where\footnote{We let $C^\infty(M)_{(i)} = C^\infty(M)$ for $i\leq 0$.} 
    \[\gr(C^\infty(M)) = \bigoplus_{i\in \Z} C^\infty(M)_{(i)}/ C^\infty(M)_{(i+1)}.  \]
\cite[Theorem 4.2]{LOIZ-MEIN2020} states that there is a unique $C^\infty$-structure on $\nuw(M,N)$ such that any class $[f] \in \gr(C^\infty(M))$ defines, by evaluation, a smooth map $\nuw(M,N) \to \R$. The weighted deformation space is the character spectrum 
    \[ \defw(M,N) = \aHom(\rees(C^\infty(M)), \R),   \]
where 
    \[ \rees(C^\infty(M)) = \left\{ \sum_{i\in \Z} f_iz^{-i} : f_i \in C^\infty(M)_{(i)}\right\} \sset C^\infty(M)[z,z^{-1}]  \]
is the Rees algebra of the filtered algebra $C^\infty(M)$. \cite[Theorem 4.2]{LOIZ-MEIN2020} states that $\defw(M,N)$ has a unique $C^\infty$-structure such that, for any $f\in C^\infty(M)_{(i)}$, evaluation at $fz^{-i}$ defines a smooth map $\defw(M,N)\to \R$. In particular, evaluation at $z \in \rees(C^\infty(M))$ defines a surjective submersion $\pi_\delta : \defw(M,N) \to \R$, and the fibres of $\pi_\delta$ are given by  
        \[ \pi_\delta^{-1}(t) = \left\{
        \begin{array}{ll}
            M & t\neq 0,  \\
            \nuw(M,N) & t = 0.  
        \end{array}
    \right.\]

\subsection{Definition of linear weightings}

A linear weighting of a vector space $V$ is a filtration of $V$ by subspaces. Therefore, we define a weighted vector bundle to be one which is locally the product of a weighted manifold with a filtered vector space. This motivates the following definition. 

Let $(M,N)$ be a weighted pair and $M \times \R^k\to M$ the trivial bundle of rank $k$ over $M$. A \emph{vertical weight vector} for $M \times \R^k$ is a $k$-tuple of integers $(v_1, \dots, v_k) \in \Z^k$. A choice of vertical weight vector determines a filtration of $\G(M \times \R^k)$ by $C^\infty(M)$-submodules
    \begin{equation}
    \label{equation: local model for linear weighting}
        \G(M\times \R^k)_{(i)} = \sum_{a=1}^k C^\infty(M)_{(i-v_a)}\sigma_a, 
    \end{equation}
where $\sigma_1, \dots, \sigma_k$ is the standard basis for $\R^k$.  
	
\begin{definition}
\label{definition: linear weighting}
    A \emph{linear weighting} of a rank $k$ vector bundle $V$ over the weighted pair $(M,N)$ is a $\Z$-graded filtration 
        \[ \cdots \supseteq \G_{V, (i)} \supseteq \G_{V,(i+1)} \supseteq \cdots \]
    of the sheaf of sections $\G_V$ by $C^\infty_M$-submodules such that for every point $p\in M$ there exists an open neighbourhood $U\sset M$ containing $p$ and a frame $\sigma_1, \dots, \sigma_k \in \G(V|_U)$ such that $\G(V|_U)_{(i)}$ is given by~\eqref{equation: local model for linear weighting}. The frame $\sigma_a$ is called a \emph{weighted frame}. We refer to a vector bundle with a linear weighting as a \emph{weighted vector bundle}.
\end{definition}

\begin{remarks}
    \begin{enumerate}
        \item If $v_1, \dots, v_k$ is the vertical weight sequence for $V$, then  $\G_{V, (i)} = \G_V$ for $i \leq \min_a\{v_a\}$.

        \item  By definition, $\G_V$ is a filtered module over the filtered algebra $C^\infty_M$. That is, for all $i, j \in \Z$,
            \[ C^\infty_{M,(i)}\cdot \G_{V, (j)} \sset \G_{V, (i+j)}.\]

        \item Using that that we are working in the $C^\infty$-category again, and work instead with the filtration of global sections  
            \[ \cdots \supseteq \G(V)_{(i)} \supseteq \G(V)_{(i+1)} \supseteq \cdots \]
    \end{enumerate}
\end{remarks}

\begin{proposition}
\label{proposition: linear weighting determines filtration by subbundles}
    Let $(M,N)$ be a weighted pair. A linear weighting of $V\to M$ with vertical weights $v_1, \dots, v_k$ determines a filtration of $V|_N$  
        \[ \cdots \supseteq (V|_N)_{(i)} \supseteq (V|_N)_{(i+1)} \supseteq \cdots  \]
    by wide subbundles $(V|_N)_i$ of rank $k_i = \#\{ a : v_a \geq i\}$ with the property that $\G((V|N)_i)$ is given by the image of $\G(V)_{(i)}$ in 
        \[\G(V)/(\van{N}\cdot \G(V)) = \G(V|_N) \]
    under the quotient map. 
\end{proposition}
\begin{proof}
    Let $p\in N$ be contained in the open neighbourhood $U\sset M$, and let $\sigma_a$ be a weighted frame for $V|_U$. Then the image of $\G(V|_U)_{(i)}$ in the quotient $\G(V|_U)/(\van{N\cap U}\cdot \G(V|_U))$ is freely generated by 
        \[ \{ \sigma_a|_{N\cap U} : v_a \geq i \}. \qedhere \]
\end{proof}

\begin{example}
\label{example: weighting defined by subbundles}
    There is a unique weighting of $M$ along itself; it is given by 
        \[ C^\infty(M)_{(i)} = \left\{
            \begin{array}{ll}
                C^\infty(M) & i = 0,  \\
                0 & i > 0.
            \end{array}
        \right.  \]
    If $M$ is weighted along itself, then a linear weighting of $V\to M$ is given by a filtration of $V$
        \[ \cdots \supseteq V_i \supseteq V_{i+1} \supseteq \cdots  \]
    by wide subbundles, in the sense that 
        \[ \G(V)_{(i)} = \G(V_{i}).  \]
\end{example}

\begin{example}
\label{example: tangent weighting}
    Suppose that $(M,N)$ is a weighted pair. Define $\mathfrak{X}(M)_{(i)}$ to be the vector fields that shift filtration degree of $C^\infty(M)$ by $i$; that is, $X\in \mathfrak{X}(M)_{(i)}$ if and only if 
        \[ f\in C^\infty(M)_{(j)} \implies Xf \in C^\infty(M)_{(i+j)}. \]
    If $x_a$ is a local weighted coordinate system on $M$, then the local coordinate vector fields $\frac{\bd}{\bd x_a}$ define a local weighted frame for $TM$. Thus, the filtration  
        \[ \cdots \supseteq \mathfrak{X}(M)_{(i)} \supseteq \mathfrak{X}(M)_{(i+1)} \supseteq \cdots  \]
    defines a linear weighting of $TM$.
\end{example}

\begin{example}
\label{example: trivial weighting along subbundle}
    Let $N\sset M$ be a closed submanifold with vanishing ideal $\van{N}$. The trivial weighting of $M$ along $N$ is given by order of vanishing, 
        \[ C^\infty(M)_{(j)} = \van{N}^j.  \]    
    If $W\to N$ is a subbundle of $V\to M$, then the filtration 
        \[ \G(V)_{(i)} = \left\{
            \begin{array}{ll}
                 \G(V) & i \leq -1,  \\
                \G(V, W) & i = 0, \\
                \van{N}^i\cdot \G(V) & i \geq 1, 
            \end{array}
        \right.\]
    determines a linear weighting of $V$ such that the induced filtration of $V|_N$ is given by 
        \[ V|_N \supseteq W \supseteq 0.  \]
    If $M$ is trivially weighted along $N$, this recovers the tangent weighting of $TM$ from the previous example for $W = TN$. 
\end{example}

\begin{example}
    Let $\R$ be given the trivial weighting along the origin, so that the coordinate $x$ has weight $1$. Let $V = \R \times \R^2 \corr{{\rm pr}_1} \R$ be the trivial bundle of rank 2 over $\R$, and let $\sigma_1, \sigma_2 \in \G(V)$ be standard frame. The linear weightings
	\begin{align*}
            \G(V)_{(i)} & = C^\infty(\R)_{(i)}\cdot  \sigma_1 + C^\infty(\R)_{(i+2)}\cdot \sigma_2 \\
            \G(V)'_{(i)} & = C^\infty(\R)_{(i)}\cdot (\sigma_1 + x\sigma_2) + C^\infty(\R)_{(i+2)}\cdot \sigma_2 
	\end{align*}
    determine the same filtration of $V_{0} = \{0\}\times \R^2$, but they are different weightings. Indeed, $\sigma_1+x\sigma_2 \in \G(V)'_{(0)}$, but $\sigma_1+x\sigma_2 \notin\G(V)_{(0)}$. In particular, this example shows that the filtration of $V|_N$ alone is not enough to recover the weighting of $V$.
\end{example}

\subsection{Constructions}

We now take some time to explain how various constructions with vector bundles work in the weighted setting.

\subsubsection{Dual weighting}

The dual of a weighted vector bundle is linearly weighted by 
        \[ \tau \in \G(V^*)_{(i)} \iff \forall \sigma\in \G(V)_{(j)},\ \la \tau, \sigma \ra \in C^\infty(M)_{(i+j)}. \]
If $\sigma_a$ is a weighted frame for $V|_U$ then the corresponding dual frame $\tau_a$ is a weighted frame for $V^*|_U$. In particular, If $v_a$ are the vertical weights for $V$ then $-v_a$ are the vertical weights for $V^*$ and $V=(V^*)^*$ as weighted vector bundles. Furthermore, the corresponding filtration of $V^*|_N$ is given by 
    \[ (V^*|_N)_{i} = {\rm ann}((V|_N)_{-i+1}). \]

\begin{example}
    If $(M,N)$ is a weighted pair then $T^*M$ obtains a natural linear weighting given by the dual weighting of $TM$.
\end{example}

\subsubsection{Direct sums, tensor products, etc.}

If $V\to M$ and $W\to M$ are weighted vector bundles over the weighted pair $(M, N)$, then $V\oplus W$, $V\otimes W$, ${\rm Hom}(V,W)$, $\wedge^k V$, and ${\rm Sym}^k(V)$ all inherit linear weightings in a canonical way. 

For example, the linear weighting on $V\otimes W$ is given by 
    \[ \G(V\otimes W)_{(k)} = \sum_{i+j = k}\G(V)_{(i)} \otimes_{C^\infty(M)} \G(W)_{(j)}.   \]
The linear weighting on ${\rm Hom}(V, W)$ is given by the identification ${\rm Hom}(V,W) \cong V^*\otimes W$. With respect to this weighting we have that, for all $i, j \in \Z$,
    \[ \G({\rm Hom}(V, W))_{(i)} \times \G(V)_{(j)} \to \G(W)_{(i+j)}. \]

\subsubsection{Shifted weighting} Given a linear weighting of $V\to M$ we denote by $V[k]$ the vector bundle $V$ linearly weighted by 
    \[ \G(V[k])_{(i)} = \G(V)_{(i+k)}. \]

\begin{example}
    For any linearly weighted $V$ we have that 
        \[ V[k]^* = V^*[-k].  \] 
\end{example}

\subsubsection{Pullback bundles}

Suppose that $(M,N)$ and $(M', N')$ are weighted pairs and that $\varphi:M' \to M$ is a weighted morphism. 

\begin{proposition}
    If $V\to M$ is a weighted vector bundle, then 
        \begin{equation}
        \label{equation: pullback weighting}   
            \G(\varphi^*V)_{(i)} = \sum_{j\geq 0}C^\infty(M')_{(j)}\cdot \varphi^*\G(V)_{(i-j)}. 
        \end{equation} 
    defines a linear weighting of the pullback bundle $\varphi^*V\to M'$. 
\end{proposition}
\begin{proof}
    Let $U\sset M$ be open and let $\sigma_1, \dots, \sigma_k$ be a weighted frame for $V|_U$. I claim that the pullbacks $\varphi^*\sigma_a \in \G(\varphi^*(V|_{U}))$ define a weighted frame for $\varphi^*(V|_U)$. Indeed, we have that 
        \begin{align*}
            \G(\varphi^*(V)|_{\varphi^{-1}(u)})_{(i)} & = \sum_{j\geq 0} C^\infty(\varphi^{-1}(U))_{(j)}\cdot\varphi^*\G(V|_U)_{(i-j)} \\
            & = \sum_a \left( \sum_{j\geq 0}C^\infty(\varphi^{-1}(U))_{(j)}\cdot \varphi^*C^\infty(U)_{(i-j-v_a)}\right)\varphi^*\sigma_a.
        \end{align*}
    To complete the proof, we must show that
        \[ \sum_{j\geq 0}C^\infty(\varphi^{-1}(U))_{(j)}\cdot \varphi^*C^\infty(U)_{(i-j-v_a)} = C^\infty(\varphi^{-1}(U))_{(i-v_a)}; \]
    note that there is nothing to show if $i-v_a < 0$. Since $\varphi:M'\to M$ is a weighted morphism, it follows that $\varphi^*C^\infty(U)_{(i-j-v_a)} \sset C^\infty(\varphi^{-1}(U))_{(i-j-v_a)}$, hence the left hand side is contained in the right hand side. The other inclusion follows by putting $j = i-v_a$ in the sum. 
\end{proof}

\subsection{Linear weightings as filtrations of polynomial functions} 

In order to define the weighted normal bundle and weighted deformation bundle of a weighted vector bundle $V\to M$, it will be convenient to have a characterization of linear weightings in terms of polynomial functions. To this end, let $V\to M$ be a weighted vector bundle and let $U\sset M$ be open. For $n\in \N$, let 
    \[ C^\infty_{[n]}(V|_U)_{(i)} = \{ f\in C^\infty_{[n]}(V|_U) : i+nj>0,\ \sigma \in \G(V|_U)_{(j)} \implies f\circ \sigma \in C^\infty(V|_U)_{(i+nj)} \} \]
and define 
    \[ C^\infty_{pol}(V|_U)_{(i)} = \bigoplus_{n\geq 0} C^\infty_{[n]}(V|_U)_{(i)}.  \]
This defines a multiplicative filtration of the sheaf of polynomial functions on $V$ by $C^\infty_M$-submodules. Note that, by definition, 
    \begin{equation}
    \label{equation: polynomial and symmectric algebra equivalence}
        C^\infty_{[n]}(V|_U)_{(i)} = \G({\rm Sym}^n(V^*|_U))_{(i)}.
    \end{equation}
Let $w = (w_1, \dots, w_m)$ and $v = (v_1, \dots, v_k)$ be the weight and vertical weight vectors for $M$ and $V$, respectively, let $x_1, \dots x_m \in C^\infty(U)$ be weighted coordinates, and let $\sigma_1, \dots, \sigma_k \in \G(V|_U)$ be a weighted frame. If $y_1, \dots, y_k \in C^\infty_{[1]}(V|_U)$ denote the dual coordinates to the frame $\sigma_1, \dots, \sigma_k$, then the identification~\eqref{equation: polynomial and symmectric algebra equivalence} implies that $C^\infty_{pol}(V|_U)_{(i)}$ is generated as a $C^\infty(U)$-module by monomials of the form     
    \begin{equation}
    \label{equation: local model for weighting}
        x^\alpha y^\beta = x_1^{\alpha_1}\cdots x_m^{\alpha_m}y_1^{\beta_1}\cdots y_k^{\beta_k} \quad \text{with } \alpha\cdot w - \beta\cdot v \geq i;
    \end{equation}
note that the vertical coordinates $y_1, \dots, y_k$ have weights $-v_1, \dots, -v_k$.

Conversely, suppose that we have a $C^\infty(M)$-module filtration 
    \[ \cdots \supseteq C^\infty_{pol, V, (i)} \supseteq C^\infty_{pol, V, (i+1)} \supseteq \cdots  \]
of the sheaf $C^\infty_{pol, V}$ of $C^\infty_M$-modules with the property that locally one can find vector bundle coordinates $x_1, \dots, x_m \in C^\infty(U)$, $y_1, \dots, y_k\in C^\infty_{[1]}(V|_U)$ so that $C^\infty_{pol}(V|_U)_{(i)}$ is generated as a $C^\infty(U)$-module by monomials of the form     
        \begin{equation}
            x^\alpha y^\beta = x_1^{\alpha_1}\cdots x_m^{\alpha_m}y_1^{\beta_1}\cdots y_k^{\beta_k} \quad \text{with } \alpha\cdot w - \beta\cdot v \geq i.
        \end{equation}
For $U\sset M$ open, define
    \begin{equation}
    \label{equation: induced weighting on M}
        C^\infty(U)_{(i)} = C^\infty_{[0]}(V|_U) \cap C^\infty_{pol}(V|_U)_{(i)}
    \end{equation}
and
    \begin{equation}
    \label{equation: induced filtration on sections}
        \G(V|_U)_{(i)} = \{ \sigma \in \G(V|_U) : i+j > 0,\ f\in C^\infty_{[1]}(V|_U)_{(j)} \implies f\circ \sigma \in C^\infty(U)_{(i+j)}\}.
    \end{equation}
The filtrations~\eqref{equation: induced weighting on M} and~\eqref{equation: induced filtration on sections} define a weighting of $M$ and a linear weighting of $V$, respectively. These constructions are clearly inverse to one another, and so we have established the following. 

\begin{theorem}
 \label{theorem: linear weightings in terms of polynomials}
     Let $V$ be a rank $k$ vector bundle over the $m$-dimensional manifold $M$. There is a one to one correspondence between linear weightings of $V$ and multiplicative filtrations 
        \[ \cdots \supseteq C^\infty_{pol}(V)_{(i)} \supseteq C^\infty_{pol}(V)_{(i+1)} \supseteq \cdots \]
    of the sheaf of polynomial functions on $V$ with the following property: there exist tuples $(w_1, \dots, w_m) \in \Z^m_{\geq 0}$ and $(v_1, \dots, v_k) \in \Z^k$ such that for each $p\in V$, there is an open set $U\sset M$ with $p\in V|_U$ and vector bundle coordinate functions $x_a \in C^\infty_{[0]}(V|_U)$, $y_b\in C^\infty_{[1]}(V|_U)$ such that $C^\infty_{pol}(V|_U)_{(i)}$ is generated as a $C^\infty(U)$-module by the monomials 
        \[ x^sy^t = x_1^{s_1}\cdots x_m^{s_m}\cdot y_1^{t_1} \cdots y_k^{t_k} \]
    such that $s\cdot w - t\cdot v = \sum_{i=1}^ms_iw_1 - \sum_{i=1}^kt_iv_i \geq i$. 
\end{theorem}

We refer to the coordinates in the above theorem as \emph{weighted vector bundle coordinates}. 

\begin{example}
    If $W\to N$ is a vector subbundle of $V\to M$, then order of vanishing defines a linear weighting of $V$, which we refer to as the trivial weighting along $W$. This agrees with the linear weighting defined in Example~\ref{example: trivial weighting along subbundle}.
\end{example}

\begin{remark}
    Let $V\to M$ be a linearly weighted vector bundle. Recall that the weighting of $V$ determines a filtration
        \[ \cdots \supseteq (V|_N)_i \supseteq (V|_N)_{i+1} \supseteq \cdots \]
    of $V|_N$ by wide subbundles. In local vector bundle coordinates $x_1, \dots, x_m \in C^\infty(U)$, $y_1, \dots, y_k\in C^\infty_{[1]}(V|_U)$, $(V|_{N\cap U})_{i}$ is cut out by the equations 
        \begin{align*}
            x_a = 0 & \text{ for } w_a > 0,  \\
            y_b = 0 & \text{ for } v_b < i.
        \end{align*}
\end{remark}

\subsubsection{Weighted subbundles} 
\label{subsection: weighted subbundles}

Now that we have a characterization of linear weightings in terms of polynomial functions, we can define weighted subbundles. 

\begin{definition}
    Let $V\to M$ be a weighted vector bundle. A subbundle $W\to R$ is a \emph{weighted subbundle} if there exists a weighted atlas of subbundle coordinates for $W$. 
\end{definition}

That is, $W$ is a weighted subbundle if and only if for every $w\in W$ there exist weighted vector bundle coordinates $x_a\in C^\infty_{[0]}(V|_U)$ and $y_b \in C^\infty_{[1]}(V|_U)$ such that $W$ is locally defined by the vanishing of a subset of the coordinates. Coordinates with this property are called \emph{weighted subbundle coordinates}.

\begin{remark}
    Suppose that $V\to M$ is a linearly weighted vector bundle and the induced weighting of $M$ is along $N$. If $W\to R$ is a weighted subbundle, then necessarily we have that $R$ and $N$ intersect cleanly. 
\end{remark}

\begin{example}
    If $V$ is a linearly weighted vector bundle over the weighted pair $(M,N)$, then $V|_N$ is a weighted subbundle. Any system of weighted vector bundle coordinates are weighted subbundle coordinates for $V|_N$.
\end{example}

\begin{example}
    If $V$ is linearly weighted vector space then any subspace is a weighted subspace. 
\end{example}

\begin{proposition}
    Let $V \to M$ be a linearly weighted vector bundle. If $W \to R$ is a weighted subbundle, then $W$ inherits a linear weighting. The induced weighting of $R$ is along along $R\cap N$. 
\end{proposition}
\begin{proof}
    The weighting of $W$ is given by 
        \[ C^\infty_{pol}(W)_{(i)} = C^\infty_{pol}(V)_{(i)}/(C^\infty_{pol}(V)_{(i)} \cap \van{W}), \]
    where $\van{W}$ denotes the vanishing ideal of $W$. To see that this defines a weighting, we note that the restriction of weighted subbundle coordinates give weighted vector bundle coordinates for $W|_{U\cap R}$. 
\end{proof}

\subsection{Weighted VB-morphisms}

We close this sections be discussing the morphisms of weighted vector bundles. The easiest definition uses the characterization of linear weightings in terms of polynomial functions. 

\begin{definition}
    A vector bundle morphism $\varphi:V\to V'$ between linearly weighted vector bundles $V\to M$ and $V'\to M'$ is called \emph{weighted} if $\varphi^*C^\infty_{pol}(V')_{(i)} \sset C^\infty_{pol}(V)_{(i)}$ for all $i\in \Z$. 
\end{definition}

Recall that a vector bundle map $\varphi:V\to V'$ induces a module map $\varphi^*:\G((V')^*) \to \G(V^*)$. In terms of this perspective, one has the following characterization of weighted vector bundle morphisms. 
	
\begin{proposition}
    Let $\varphi:V\to V'$ be a vector bundle morphism between linearly weighted vector bundles $V\to M$ and $V'\to M'$, respectively. The $\varphi$ is weighted if and only if 
        \begin{enumerate}
            \item the base map $\varphi_M:M\to M'$ is weighted and 
            \item for all $i\in \Z$, $\varphi^* : \G((V')^*)_{(i)} \to \G(V^*)_{(i)}$.
        \end{enumerate}
\end{proposition}
\begin{proof}
    Suppose that $\varphi:V \to V'$ is a weighted vector bundle morphism. Then, 
        \begin{align*}
            \varphi_M^*C^\infty(M')_{(i)} & = \varphi_M^*(C^\infty_{pol}(V')_{(i)} \cap C^\infty_{[0]}(W)) \\
            & \sset C^\infty_{pol}(V)_{(i)} \cap C^\infty_{[0]}(V) = C^\infty(M)_{(i)},
        \end{align*}
    whence $\varphi_M:M\to M'$ is weighted. Similarly, for any $\sigma \in \G((V')^*)$, let $f_\sigma \in C^\infty_{[1]}(V)$ be the corresponding function under the identification $C^\infty_{[1]}(V) = \G(V^*)$. Then
        \[ \varphi^*(\sigma) = \varphi^*f_\sigma \in C^\infty_{[1]}(V)_{(i)} = \G(V^*)_{(i)}.  \]

    The converse follows by a similar argument, using that $\varphi^*: C^\infty_{pol}(V')\to C^\infty_{pol}(V)$ is an algebra morphism and $C^\infty_{pol}(V)_{(i)}$ is locally generated as a filtered algebra by $C^\infty(M)$ and $C^\infty_{[1]}(V) = \G(V^*)$. 
\end{proof}

\begin{example}
    If $f:M\to M'$ is a weighted morphism between weighted pairs $(M, N)$ and $(M', N')$, then the tangent map $Tf:TM \to TM'$ is a weighted vector bundle morphism. 
\end{example}
	
In the case that $V$ and $V'$ are vector bundles over a common base the definition of a weighted vector bundle morphism reduces to the one that one might expect. 
	
\begin{proposition}
    Suppose that that $V$ and $V'$ and linearly weighted vector bundles over a common weighted pair $(M, N)$ and $\varphi : V\to V'$ is a vector bundle morphism covering the identity. Then $\varphi$ is a weighted vector bundle morphism if and only if $\varphi(\Gamma(V)_{(i)}) \sset \Gamma(V')_{(i)}$ for all $i$. 
\end{proposition}
\begin{proof}
    Suppose that $\varphi^*(\Gamma((V')^*)_{(i)}) \sset \Gamma(V^*_{(i)})$, and let $\sigma \in \Gamma(V)_{(i)}$. Let $\tau \in \Gamma((V')^*)_{(j)}$ be arbitrary, and note that 
		\[ \la \tau, \varphi(\sigma) \ra = \la \varphi^*(\tau), \sigma \ra \in C^\infty(M)_{(i+j)},  \]
    since $\varphi^*(\tau)\in \Gamma(V^*)_{(j)}$; here the angular brackets denote the pairing between a vector bundle and its dual. This implies that $\varphi(\sigma)\in \Gamma(V')_{(i)}$, proving one inclusion. The opposite inclusion follows similarly.
\end{proof}

\section{The weighted normal bundle}

We now define the weighted normal bundle of a linearly weighted vector bundle. Since we have defined weightings in terms of filtrations of smooth functions on $V$ which are polynomial in the fibres, much of the discussion in~\cite[Section 4]{LOIZ-MEIN2020} can be carried over, almost verbatim, to our setting. 

\subsection{Definition}

Let $V\to M$ be a linearly weighted vector bundle. Let $\gr(C^\infty_{pol}(V))$ be the graded algebra associated to the filtered algebra, with graded components given by 
    \[ \gr(C^\infty_{pol}(V))_{i} = C^\infty_{pol}(V)_{(i)}/C^\infty_{pol}(V)_{(i+1)}.  \]

\begin{definition}
    The \emph{weighted normal bundle} of the linearly weighted vector bundle $V\to M$ is the character spectrum 
        \[ \nuw(V) = \aHom(\gr(C^\infty_{pol}(V)), \R) \]
\end{definition}

Given a function $f\in C^\infty_{pol}(V)_{(i)}$, let $f^{[i]}$ denote the class of $f$ in $\gr(C^\infty_{pol}(V))_{i}$. We think of $f^{[i]}$ as a function on $\nuw(V)$, defined by evaluation
    \begin{align*}
        f^{[i]} : \nuw(V)  \to \R, \quad \varphi  \mapsto \varphi(f^{[i]}). 
    \end{align*}
Let $\pi:V\to M$ and $\kappa_\lambda: V \to V$ denote the bundle projection and scalar multiplication by $\lambda \in \R$, respectively. Since these maps are filtration preserving, they induce maps 
    \[ \nuw(\pi) : \nuw(V) \to \nuw(M,N) \quad \text{and} \quad \nuw(\kappa_\lambda) : \nuw(V) \to \nuw(V). \]
The proof of~\cite[Theorem 4.2]{LOIZ-MEIN2020} carries over almost verbatim to our setting to give the following theorem. 

\begin{theorem}
    Let $V\to M$ be a weighted vector bundle. The weighted normal bundle $\nuw(V)$
    has the unique structure of a $C^\infty$-vector bundle over $\nuw(M,N)$ of rank equal to that of $V$, such that for all $n\geq 0$ 
        \[\gr(C^\infty_{[n]}(V)) \sset C^\infty_{[n]}(\nuw(V)). \]
    Given weighted vector bundle coordinates $x_a$ and $y_b$ on $V|_U$, the functions $x_a^{[w_a]}$, $y_b^{[v_b]}$ serve as vector bundle coordinates on $\nuw(V|_U) = \nuw(V)|_{\nuw(U, U\cap N)}$. Moreover, this construction is functorial: any weighted vector bundle morphism $\varphi:V\to W$ between linearly weighted vector bundles defines a vector bundle morphism  $\nuw(\varphi):\nuw(V)\to \nuw(W)$.
\end{theorem}  

The vector bundle projection on $\nuw(V)$ given by $\nuw(\pi)$ and scalar multiplication by $\lambda \in \R$ is given by $\nuw(\kappa_\lambda)$,

\begin{example}
    Suppose that $W\to N$ is a subbundle of $V \to M$. If $V$ is given the trivial weighting along $W$, then 
        \[ \nuw(V) = \nu(V, W),  \]
    as a vector bundle over $\nu(M,N)$. 
\end{example}

\begin{example}
    If $V$ is linearly weighted by a filtration of wide subbundles 
        \[ V = V_{-r} \supseteq V_{-r+1} \supseteq \cdots \supseteq V_{-1} \supseteq 0, \]
    as in Example~\ref{example: weighting defined by subbundles}, then $\nuw(V) = \gr(V) \to M$. In particular, if $V = V_{-r}\oplus \cdots \oplus V_{-1}$ is a graded vector bundle, then $\nuw(V) = V$.
\end{example}

\begin{example}
    If $(M,N)$ is a weighted pair and $TM$ is linearly weighted as in Example~\ref{example: tangent weighting}, then 
        \[ \nuw(TM) = T\nuw(M,N). \]
\end{example}

\begin{remark}
    The weighted normal bundle is invariant under shiftings of the weighting: for any $k \in \Z$ one has that 
        \[ \nuw(V[k]) = \nuw(V). \]
    In particular, we can shift the weighting of $V$ so that all the vertical weights are non-negative, and this allows us to realize the weighted normal bundle as a bundle of homogeneous spaces of nilpotent Lie groups (cf.~\cite[Proposition 7.7]{LOIZ-MEIN2020}.)
\end{remark}

\subsection{Sections of the weighted normal bundle}
\label{subsection: sections of the normal bundle}

We now want to understand how sections of the weighted normal bundle are related to the associated graded module 
    \[ \gr(\G(V)) = \bigoplus_{i\in \Z} \G(V)_{(i)}/\G(V)_{(i+1)}.  \]
Given $\sigma \in \G(V)_{(i)}$, let $\sigma^{[i]}$ denote its class in $\G(V)_{(i)}/\G(V)_{(i+1)}$. Recall that for any $f \in C^\infty_{[n]}(V)_{(j)}$ one has that $f\circ \sigma \in C^\infty(M)_{(j+ni)}$. 

\begin{definition}
    Given $\sigma \in \G(V)_{(i)}$, the \emph{$i$-th homogeneous approximation} is the map $ \sigma^{[i]} : \nuw(M,N) \to \nuw(V)$ defined by 
        \begin{equation}
        \label{equation: homogeneous approximation definition}
            (\sigma^{[i]}(\varphi))(f^{[j]}) = \varphi((f\circ \sigma)^{[j+ni]}),
        \end{equation}
    where $\varphi \in \nuw(M,N)$ and $f \in C^\infty_{[n]}(V)_{(j)}$.
\end{definition}

\begin{lemma}
\label{lemma: graded modules gives sections of weighted normal bundle}
    For $\sigma \in \G(V)_{(i)}$, the $i$-th homogeneous approximation is a smooth section of $\nuw(V)$ which depends only on the class of $\sigma$ in $\G(V)_{(i)}/\G(V)_{(i+1)}$. Moreover, for any $f\in C^\infty_{[n]}(V)_{(j)}$ and $g\in C^\infty(M)_{(k)}$ one has 
        \begin{equation}
        \label{equation: homogeneous approximation properties}
            f^{[j]}\circ \sigma^{[i]} = (f\circ \sigma)^{[j+ni]} \in C^\infty(\nuw(M,N)) \quad \text{and} \quad g^{[k]}\sigma^{[i]} = (g\sigma)^{[i+k]} \in \G(\nuw(V)).
        \end{equation}
\end{lemma}
\begin{proof}
    The equation $f^{[j]}\circ \sigma^{[i]} = (f\circ\sigma)^{[j+ni]}$ follows from~\eqref{equation: homogeneous approximation definition}, and it follows from this that $\sigma^{[i]}:\nuw(M,N) \to \nuw(V)$ is smooth. To see that it is a section, we note that 
        \[ \nuw(\pi)\circ \sigma^{[i]} = \nuw(\pi\circ \sigma) = {\rm id}_{\nuw(M,N)}. \]
    Given $g \in C^\infty(M)_{(k)}$, $f\in C^\infty_{[n]}(V)_{(j)}$ and $\varphi \in \nuw(M,N)$, using that $f$ is homogeneous of degree $n$ we compute 
        \begin{align*}
            ((g^{[k]}\sigma^{[i]})(\varphi))(f^{[j]}) & = (\nuw(\kappa_{\varphi(g^{[k]})})\sigma^{[i]}(\varphi))(f^{[j]}) = (\sigma^{[i]}(\varphi))(\kappa_{\varphi(g^{[k]})}^*f^{[j]}) \\
            & = (\sigma^{[i]}(\varphi))(\varphi(g^{[k]})^nf^{[j]}) = \varphi(g^{[k]})^n\varphi((f\circ \sigma)^{[j+ni]}) \\
            & = \varphi((g^n)^{[nk]}(f\circ \sigma)^{[j+ni]}) = \varphi((g^n(f\circ \sigma))^{[j+ni+nk]}) \\
            & = \varphi((f\circ (g\sigma))^{[j+n(i+k)]}) = ((g\sigma)^{[i+k]}(\varphi))(f^{[j]}),
        \end{align*}
    which establishes~\eqref{equation: homogeneous approximation properties}. To show that $\sigma^{[i]}$ only depends on the class of $\sigma \in \G(V)_{(i)}/\G(V)_{(i+1)}$, it suffices to show that if $\sigma \in \G(V)_{(i+1)}$ then $\sigma^{[i]} = 0$. Specifically, what this means is that for any $f\in C^\infty_{pol}(V)_{(j)}$ and $\varphi \in \nuw(M,N)$ one has
        \begin{equation}
        \label{equation: what the zero section is}
            (\sigma^{[i]}(\varphi))(f^{[j]}) = \varphi(\kappa_0^*f^{[j]}).
        \end{equation}
    If $f\in C^\infty_{[n]}(V)_{(j)}$ with $n\geq 1$, one has that 
        \[ f\circ \sigma \in C^\infty(M)_{(j+n(i+1))} \implies (f\circ \sigma)^{[j+ni]} = 0, \]
    so~\eqref{equation: what the zero section is} follows in this case. For $f\in C^\infty_{[0]}(V)_{(j)} = C^\infty(M)_{(j)}$ one has that 
        \[ (\sigma^{[i]}(\varphi))(f^{[j]}) = \varphi((f\circ \sigma)^{[j]}) = \varphi(f^{[j]}) =  \varphi(\kappa_0^*f^{[j]}), \]
    since $f$ is homogeneous of degree zero. Since any polynomial is a sum of monomials, it follows that $\sigma^{[i]} = 0$.  
\end{proof}

This establishes a map $\gr(\G(V)) \to \G(\nuw(V))$. To see that it is non-trivial, we have the following proposition. 

\begin{theorem}
    If $\sigma_b$ is a weighted frame for $V|_U$, then the homogeneous approximations 
    $\sigma_b^{[-v_b]}$ define a frame for $\nuw(V|_U) = \nuw(V)|_{\nuw(U, U\cap N)}$. In particular, we have that 
        \[ \G(\nuw(V)) = C^\infty(\nuw(M,N))\otimes_{\gr(C^\infty(M))}\gr(\G(V)).  \]
\end{theorem}
\begin{proof}
Let $y_b \in C^\infty_{[1]}(V|_U)_{(v_b)}$ be the linear vector bundle coordinates defined by the frame $\sigma_a$ and note that for any non-zero $\varphi_0 \in \nuw(U, U\cap N)$ one has 
    \begin{align*}
    \label{equation: vector bundle coordinate pairing}
    \begin{split}
            \sigma_b^{[-v_b]}(\varphi_0)\left(y_a^{[v_a]}\right) = \varphi_0 \left((y_a(\sigma_b))^{[v_b-v_a]}\right) = \left\{
            \begin{array}{cc}
                 1 & \text{if $a=b$} \\
                 0 & \text{else,}
            \end{array}
        \right.
    \end{split}
    \end{align*}
which shows that $\sigma_a^{[-v_a]}$ form a linearly independent set. 
    
Since $\gr(C^\infty_{pol}(V|_U))$ is generated as an algebra over $\gr(C^\infty(U))$ by the elements $y_b^{[v_b]}$, any $\varphi \in \nuw(V|_U)$ is determined by its value on $\gr(C^\infty(U))$ and $p_b^{[v_b]}$, $b=1, \dots, k$. In particular, if  $\varphi_0 = \nuw(\pi)(\varphi)$ then we have that 
        \[ \varphi = \sum_b \varphi\left(y_b^{[v_b]}\right)\sigma_b^{[-v_b]}(\varphi_0),\] 
so $\sigma_b^{[-v_b]}$ spans as well.
\end{proof}

\subsection{Properties of the weighted normal bundle}

\subsubsection{Action of $\R^\times$}

There is a smooth action 
    $\alpha : \R^\times\times \nuw(V)\to \nuw(V)$ 
on the weighted normal bundle such for any $f\in C^\infty_{pol}(V)_{(i)}$ the function $f^{[i]} \in C^\infty(\nuw(V))$ is homogeneous of degree $i$. It is related to the scalar multiplication on $\nuw(V)$ by the following proposition.

\begin{proposition}
    For any $\lambda \in \R^\times$ and $\sigma \in \G(V)_{(i)}$, the following diagram commutes
        \begin{equation*}
        \xymatrix{
            \nuw(V) \ar[r]^{\alpha_\lambda} & \nuw(V) \\
            \nuw(M,N) \ar[u]^{\sigma^{[i]}} \ar[r]_{\alpha_{\lambda}} & \nuw(M,N) \ar[u]_{\lambda^{-i}\sigma^{[i]}}
        }
        \end{equation*}
\end{proposition}
\begin{proof}
    Let $f \in C^\infty_{[n]}(V)_{(j)}$ and $\varphi \in \nuw(M,N)$. We compute 
        \begin{align*}
            (\alpha_\lambda(\sigma^{[i]}(\varphi)))(f^{[j]}) & = (\sigma^{[i]}(\varphi))(\lambda^jf^{[j]}) = \varphi(\lambda^j(f\circ \sigma)^{[j+ni]}) \\
            & = \varphi(\lambda^{j+ni}(f\circ (\lambda^{-i}\sigma))^{[j+ni]}) = (\alpha_\lambda(\varphi))((f\circ (\lambda^{-i}\sigma))^{[j+ni]}) \\
            & = (\lambda^{-i}\sigma^{[i]}(\alpha_\lambda(\varphi)))(f^{[j]}) \qedhere
        \end{align*}
\end{proof}

In particular, the equality $\nuw(V[k]) = \nuw(V)$ does \emph{not} hold when we consider these two spaces with their action of $\R^\times$. 

\begin{remark}
    If all the vertical weighted of for $V$ are non-negative, then the action of $\R^\times$ extends to a monoid action of $\R$. Using the language of~\cite{grabowski2009higher}, $\nuw(V)$ is a \emph{bigraded} bundle in this case. 
\end{remark}

\subsubsection{Weighted subbundles}

Recall that a subbundle $W\to R$ of a linearly weighted vector bundle $V\to M$ is called a weighted subbundle if there exist a weighted atlas of subbundle coordinates for $W$. In this case, it is explained in Section~\ref{subsection: weighted subbundles} that $W$ inherits a linear weighting. 

\begin{proposition}
\label{proposition: subbundles of weighted normal bundle}
    If $W\to R$ is a weighted subbundle of $V\to M$, then $\nuw(W)$ is a subbundle of $\nuw(V)$. 
\end{proposition}
\begin{proof}
    If $x_a, y_b$ are weighted subbundle coordinates for $W|_{R\cap U}$, then the homogeneous interpolations $x_a^{[w_a]}, y_b^{[v_b]} \in C^\infty(\nuw(V|_U))$ define subbundle coordinates for $\nuw(W|_{R\cap U})$. 
\end{proof}

Let $(V|_N)_i \to N$ be the filtration of $V|_N$ defined by the linear weighting (see Proposition~\ref{proposition: linear weighting determines filtration by subbundles}). Proposition~\ref{proposition: subbundles of weighted normal bundle} allows us to describe the weighted normal bundle of $V$ in terms of the associated graded bundle 
    \[ \gr(V|_N) = \bigoplus_i (V|_N)_{i}/(V|_N)_{i+1}. \]

\begin{corollary}
    Suppose that $V \to M$ is linearly weighted. If $\pi : \nuw(M,N)\to N$ denotes the graded bundle projection, then one has a (non-canonical) identification 
        \[ \nuw(V) \cong \pi^*\gr(V|_N).\]
\end{corollary}
\begin{proof}
    Recall that $V|_N$ is a weighted subbundle of $V$ and that the induced weighting of $V|_N$ is given by the filtration $(V|_N)_i$. In particular, this implies that 
        \[ \nuw(V|_N) = \gr(V|_N).  \]
    On the other hand $\nuw(V|_N)$ is a subbundle of $\nuw(V)|_N$ of equal rank, and so $\nuw(V|_N) = \nuw(V)|_N$. Since $\pi:\nuw(M, N)\to N$ is a smooth homotopy inverse of the inclusion $\iota:N\into \nuw(M,N)$ it follows that 
        \[ \nuw(V) \cong \pi^*\iota^*\nuw(V) = \pi^*(\nuw(V)|_N) = \pi^*\gr(V|_N), \]
    as claimed. 
\end{proof}

\begin{example}
    If $V\to M$ is given the trivial weighting along the subbundle $W\to R$, then this proposition amounts to the identification 
        \[ \nu(V,W) \cong \pi^*(V|_N/W \oplus W) \]
    as vector bundles over $\pi:\nu(M,N)\to N$. 
\end{example}

\section{The weighted deformation bundle}

Let $V\to M$ be a linearly weighted vector bundle. As explained above, the weighting of $M$ determines a weighted deformation space $\defw(M,N)$ which admits a set-theoretic decomposition 
    \[ \defw(M,N) = \nuw(M, N) \sqcup (M\times \R^\times).  \]
The purpose of this section is to explain how this construction carries over to the context of linearly weighted vector bundles, yielding a weighted deformation bundle $\defw(V)$ which is itself a vector bundle over $\defw(M,N)$. Since many of the proofs in this section are very similar to those in the previous section we will omit them for brevity. 

\subsection{Definitions}

Let $V\to (M,N)$ be a linearly weighted vector bundle, and let $\rees(C^\infty_{pol}(V))$ be the Rees algebra of the filtered algebra $C^\infty_{pol}(V)$,  
    \[ \rees(C^\infty_{pol}(V)) = \left\{ \sum_{i\in \Z} f_iz^{-i} : f_i \in C^\infty_{pol}(V)_{(i)} \right\} \sset C^\infty_{pol}(V)[z,z^{-1}].   \]

We define the weighted deformation bundle in the same way as Loizdes and Meinrenken, who were motivated by the work of Haj and Higson (cf.~\cite[Definition 3.1]{sadegh2018euler}), who in turn borrow from constructions in algebraic geometry. 

\begin{definition}
    The \emph{weighted deformation bundle} of the linearly weighted vector bundle $V\to (M,N)$ is the character spectrum
        \[ \defw(V) = \aHom(\rees(C^\infty_{pol}(V)), \R) \]
\end{definition}

Given $f\in C^\infty_{pol}(V)_{(i)}$, let $\WT{f}^{[i]}$ denote the function on $\defw(V)$ defined by 
    \[ \WT{f}^{[i]}(\varphi) = \varphi(fz^{-i}).  \]
Let $\pi_\delta = \WT{1}^{[-1]}: \defw(V) \to \R$; this is a surjection, inducing a set-theoretic decomposition
    \begin{equation}
    \label{equation: decomposition of deformation space}
        \defw(V) = \nuw(V) \sqcup (V\times \R^\times).
    \end{equation}
In terms of this decomposition we have, for any $f\in C^\infty_{pol}(V)_{(i)}$, that 
    \[ \WT{f}^{[i]}|_{\pi_\delta^{-1}(t)} = \left\{
        \begin{array}{ll}
            t^{-i}f & t\neq 0 \\
            f^{[i]} & t = 0.
        \end{array}
    \right.\]
Letting $\pi:V \to M$ and $\kappa_\lambda : V \to V$ be the vector bundle projection and scalar multiplication by $\lambda \in \R$, respectively, it follows that we get maps 
    \[ \defw(\pi) : \defw(V) \to \defw(M,N) \quad \text{and} \quad \defw(\kappa_\lambda) : \defw(V)\to \defw(V).  \]
The main point of this construction is that there is a natural $C^\infty$-vector bundle structure on $\defw(V)$, with bundle projection and scalar multiplication given by $\defw(\pi)$ and $\defw(\kappa)$ respectively, in such a way that each $\WT{f}^{[i]}$ is smooth and $\pi_\delta$ is a submersion (see~\cite[Theorem 5.1]{LOIZ-MEIN2020}).

\begin{theorem}
    The weighted deformation bundle has the unique structure of a $C^\infty$-vector bundle over $\defw(M,N)$ of rank equal to that of $V$ over $\defw(M,N)$ such that $\pi_\delta:\defw(V)\to \R$ is a surjective submersion and that for all $n\geq 0$ 
        \[\rees(C^\infty_{[n]}(V)) \sset C^\infty_{[n]}(\defw(V)). \]
    Given weighted vector bundle coordinates $x_a$ and $y_b$ on $V|_U$, the functions $\WT{x_a}^{[w_a]}$, $\WT{y_b}^{[v_b]}$ serve as vector bundle coordinates on $\defw(V|_U) = \defw(V)|_{\defw(U, U\cap N)}$. Moreover, this construction is functorial: any weighted vector bundle morphism $\varphi:V\to W$ between linearly weighted vector bundles defines a vector bundle morphism  $\defw(\varphi):\defw(V)\to \defw(W)$.
\end{theorem}  

The vector bundle projection is given by $\defw(\pi)$ and scalar multiplication by $\lambda \in \R$ is given by $\defw(\kappa_\lambda)$. 

\begin{example}
    If $(M,N)$ is a weighted pair and $TM$ is linearly weighted as in Example~\ref{example: tangent weighting}, then 
        \[ \defw(TM) = \bigcup_{t\in \R} T\defw(M,N)|_{t}. \]
\end{example}

\subsection{Sections of the weighted deformation bundle}

We now explain the relationship between the $\rees(C^\infty(M)$-module
    \[ \rees(\G(V)) = \left\{ \sum_{i \in \Z} \sigma_iz^{-i} : \sigma_i \in \G(V)_{(i)} \right\} \sset \G(V)[z,z^{-1}]  \]
and sections of $\defw(V)$, in a way that mirrors Section~\ref{subsection: sections of the normal bundle}.

\begin{definition}
    Given $\sigma \in \G(V)_{(i)}$, the \emph{$i$-th homogeneous interpolation} is the map $\WT{\sigma}^{[i]} : \defw(M,N) \to \defw(V)$ defined by 
        \begin{equation}
        \label{equation: homogeneous section definition}
            (\WT{\sigma}^{[i]}(\varphi))(fz^{-j}) = \varphi((f\circ \sigma)z^{-j-ni}),
        \end{equation}
    where $\varphi \in \defw(M,N)$ and $f \in C^\infty_{[n]}(V)_{(j)}$.
\end{definition}

\begin{lemma}
\label{lemma: rees modules gives sections of weighted deformation bundle}
    For $\sigma \in \G(V)_{(i)}$, the $i$-th homogeneous interpolation is a smooth section of $\defw(V)$ such that, for any $f\in C^\infty_{[n]}(V)_{(j)}$ and $g\in C^\infty(M)_{(k)}$, one has 
        \begin{equation}
        \label{equation: homogeneous interpolation properties}
            \WT{f}^{[j]}\circ \WT{\sigma}^{[i]} = \WT{f\circ \sigma}^{[j+ni]} \in C^\infty(\defw(M,N)) \quad \text{and} \quad \WT{g}^{[k]}\WT{\sigma}^{[i]} = \WT{g\sigma}^{[i+k]} \in \G(\defw(V)).
        \end{equation}
\end{lemma}

As was the case for the weighted normal bundle, this yields an inclusion $\rees(\G(V)) \sset \G(\defw(V))$. In terms of the decomposition~\ref{equation: decomposition of deformation space} one finds, using the relations~\ref{equation: homogeneous interpolation properties}, that if $\sigma \in \G(V)_{(i)}$ then 
    \begin{equation*}
        \WT{\sigma}^{[i]}|_{\pi_\delta^{-1}(t)} = \left\{
            \begin{array}{ll}
                t^{-i}\sigma & t \neq 0 \\
                \sigma^{[i]} & t = 0. 
            \end{array}
        \right. 
    \end{equation*}

Finally, we give the following desription of $\G(\defw(V))$. 
\begin{theorem}
    If $\sigma_b$ is a weighted frame for $V|_U$, then the homogeneous interpolations $\WT{\sigma}_b^{[-v_b]}$ define a frame for $\defw(V|_U) = \defw(V)|_{\defw(U, U\cap N)}$. In particular, we have that 
        \[ \G(\defw(V)) = C^\infty(\defw(M,N))\otimes_{\rees(C^\infty(M))}\rees(\G(V)).  \]
\end{theorem}

\subsection{Properties of the weighted deformation bundle}

\subsubsection{Zoom action of $\R^\times$}

The action of $\R^\times$ on $\nuw(V)$ extends smoothly to a linear action on $\R^\times$ such for any $f\in C^\infty_{pol}(V)_{(i)}$ the homogeneous interpolation $\WT{f}^{[i]} \in C^\infty(\defw(V))$ is homogeneous of degree $i$. On the open dense set $V\times \R^\times \sset \defw(V)$, this action is given by 
    \begin{equation*}
        \alpha_\lambda(v,t) = (v, \lambda^{-1}t)
    \end{equation*}
It is related to the scalar multiplication on $\defw(V)$ by the following proposition.

\begin{proposition}
    For any $\lambda \in \R^\times$ and $\sigma \in \G(V)_{(i)}$, the following diagram commutes
        \begin{equation*}
        \xymatrix{
            \defw(V) \ar[r]^{\alpha_\lambda} & \defw(V) \\
            \defw(M,N) \ar[u]^{\WT{\sigma}^{[i]}} \ar[r]_{\alpha_{\lambda}} & \defw(M,N) \ar[u]_{\lambda^{-i}\WT{\sigma}^{[i]}}
        }
        \end{equation*}
\end{proposition}

\begin{example}
    If $V =  V_{-r}\oplus \cdots \oplus V_{-1}$ is a graded vector bundle, then $\defw(V) = V \times \R \to M\times \R$ and the zoom action is given by 
        \[ \alpha_\lambda(v_{-r}, \dots, v_{-1}, t) = (\lambda^rv_{-r}, \dots, \lambda v, \lambda^{-1}t). \]
\end{example}

\section{Differential operators on weighted vector bundles}

\subsection{Weighted order}

Recall that a weighting of $M$ along $N$ determines a filtration of ${\rm DO}(M)$, where $D \in {\rm DO}(M)_{(q)}$ if and only if $D$ maps $C^\infty(M)_{(i)}$ into $C^\infty(M)_{(i+q)}$ for all $i+q \geq 1$. With respect to this filtration, one can recover the weighting as 
    \begin{equation}
    \label{equation: canonical form of a weighting}
        C^\infty(M)_{(i)} = \{ f\in C^\infty(M) : q < i,\ D\in {\rm DO}(M)_{(-q)} \implies Df|_N = 0 \}.
    \end{equation}
Similarly, if $V \to M$ is a linearly weighted vector bundle and $\DO(V)$ denotes the sheaf of linear differential operators acting on sections of $V$, the linear weighting of $V$ determines an algebra filtration of $\DO(V)$
    \[ \cdots \supseteq \DO(V)_{(q)} \supseteq \DO(V)_{(q+1)} \supseteq \cdots  \]
where
    \[ \DO(V)_{(q)} = \{ D \in \DO(V) : \sigma \in \G(V)_{(i)} \implies D\sigma \in \G(V)_{(i+q)} \}. \]
    
Given the filtration of $V|_N$, let 
    \[ \G(V, (V|_N)_{i}) = \{ \sigma \in \G(V) : \sigma|_N \in \G((V|_N)_{i}). \} \]    
We can recover the linear weighting of $V$ from the filtration of $\DO(V)$ together with the filtration of $V|_N$. 
	
\begin{proposition}
    For a linear weighting of $V$ with induced filtration of $V|_N$ given by subbundles $(V|_N)_i$, then  
        \[ \G(V)_{(i)} = \{ \sigma \in \G(V): D\in \DO(V)_{(q)} \implies D\sigma \in \G(V, (V|_N)_{i+q}), \} \]
    where $\G(V, (V|_N)_i)$ denotes the sections of $V$ which restrict to sections of $(V|_N)_i$ over $N$. 
\end{proposition}
\begin{proof}
    Suppose that $\sigma \in \G(V)_{(i)}$. By considering weighted frames, we note first of all that $\G(V)_{(i)} \sset \G(V,(V|_N)_i)$. Hence, by definition of weighted order of differential operators, if $D\in \DO(V)_{(q)}$ then $D\sigma \in \G(V)_{(i+q)} \sset \G(V, (V|_N)_{i+q})$ which proves one inclusion. 
		
    For the converse, suppose that $\sigma \in \G(V)$ has the property that $D \sigma \in \G(V, (V|_N)_{i+q})$ for all $D\in \DO(V)_{(q)}$. Choose a weighted frame $\sigma_a \in \G(V|_U)$ and write $\sigma = \sum_a f_a\sigma_a$ on $U$. Let $D\in \DO(U)_{(q)}$; the local lift of $D$ to $\G(V)$ given by 
	\[ D\sigma = \sum_a(Df_a)\sigma_a \]
    defines an inclusion $\DO(U)_{(q)} \sset \DO(V|_U)_{(q)}$, hence 
        \[ D\sigma \in \G(V|_U, (V|_{N\cap U})_{i+q}). \]
    In particular, implies that $Df_a|_N = 0$ for $v_a < i+q$. Stated differently, this says that $Df_a|_N = 0$ for all $D \in \DO(U)_{(q)}$ with $q < i-v_a$, hence $f_a \in C^\infty(U)_{(i-v_a)}$. Therefore, $\sigma \in \G(V)_{(i)}$ which completes the proof.
\end{proof}

\subsection{Weighted linearization}

Let $\gr(\DO(V))$ the associated graded algebra of the filtered algebra $\DO(V)$ with graded component 
    \[ \gr(\DO(V))_q = \DO(V)_{(q)}/\DO(V)_{(q+1)}.  \]
Given $D \in \DO(V)_{(q)}$, let $D^{[q]}$ denote its class in $\gr(\DO(V)_{q}$. This defines a differential operator on $\G(\nuw(V))$ specified by the relation
    \[ D^{[q]}\sigma^{[i]} = (D\sigma)^{[i+q]}, \]
where $\sigma \in \G(V)_{(i)}$. 

\begin{definition}
    The \emph{weighted linearization} of $D \in \DO(V)_{(q)}$ is $D^{[q]} \in \DO(\nuw(V))$. 
\end{definition}

\subsection{Interpolation of differential operators}

Any $D\in \DO(V)_{(q)}$ defines a differential operator $\WT{D}^{[q]}$ on $\defw(V)$ specified by 
    \[ \WT{D}^{[q]}\WT{\sigma}^{[i]} = \WT{(D\sigma)}^{[i+q]}.  \]
In terms of the decomposition $\defw(V) = \nuw(V)\sqcup (V\times \R^\times)$, we have that 
    \[ \WT{D}^{[q]}|_t = \left\{
        \begin{array}{cc}
            t^{-q}D & t \neq 0  \\
            D^{[q]} & t = 0. 
        \end{array}
    \right.\]
Thus, the operator $\WT{D}^{[q]}$ interpolates between the operator $D$ and its weighted symbol.

\section{The rescaled spinor bundle and \v{S}evera's algebroid}
\label{section: examples in the literature}

\subsection{The rescaled spinor bundle}

The work in this section is just a translation of the work in~\cite{higson2019spinors} to the language of linear weightings. 

Let $M$ be an even dimensional spin Riemannian manifold with spinor bundle $S\to M$ and Clifford connection $\nabla$ on $S$. Higson and Yi define a filtration of differential operators acting on sections of $S$ by declaring that Clifford multiplication and covariant differentiation have order $-1$. Thus, $D\in {\rm DO}(S)_{(-q)}$ if and only if it can be locally expressed as a sum of terms of the form 
    \[ fD_1\cdots D_q \]
where $f\in C^\infty(M)$ and each $D_i$ is either a covariant derivative $\nabla_X$, Clifford multiplication $c(X)$, or the identity operator (cf.~\cite[Definition 3.3.1]{higson2019spinors}). Higson and Yi say that $D$ has \emph{Getzler order} $-q$ if $D\in \DO(S)_{(-q)}$.

The Getzler filtration of $\DO(S)$ determines a filtration of $\G(S\boxtimes S^*)$ as follows. Recall that $\Cl(TM)$ is naturally filtered 
    \begin{equation}
    \label{equation: filtration of Clifford algebra}
        \Cl(TM) = \Cl_{-\dim(TM)}(TM) \supseteq \cdots \supseteq \Cl_0(TM) = \C,
    \end{equation}
hence $S\boxtimes S^*|_M = S\otimes S^* \cong\Cl(TM)$ is naturally filtered. Define 
    \begin{equation}
    \label{equation: getzler weighting}
        \G(S\boxtimes S^*)_{(i)} = \{ \sigma \in \G(S\boxtimes S^*) : D\in \DO(S)_{(-q)} \implies D\sigma \in \G(S\boxtimes S^*, \Cl_{i-q}(TM)) \},
    \end{equation}
where we are using the identification $\Cl(TM) = S\boxtimes S^*|_{M}$.

\begin{theorem}(\cite[Lemma 3.4.10]{higson2019spinors}).
    If $M\times M$ is given the trivial weighting along the diagonal then the filtration~\eqref{equation: getzler weighting} defines a linear weighting of $S\boxtimes S^*$ such that the induced filtration of $S\boxtimes S^*|_M$ is given by~\eqref{equation: filtration of Clifford algebra}. 
\end{theorem}

Applying the weighted deformation bundle to $S\boxtimes S^*$ with this linear weighting yields the rescaled spinor bundle. 

\begin{definition}(cf.~\cite[Section 3.4]{higson2019spinors}).
    The \emph{rescaled spinor bundle} is the weighted deformation bundle 
        \[ \mathbb{S} = \defw(S\boxtimes S^*) \to \T M. \]
\end{definition}

\begin{remark}
    Higson and Yi also show that there is a vector bundle morphism 
        \[ (\mathbb{S}\boxtimes \mathbb{S})|_{\T M^{(2)}} \to \mathbb{S} \]
    over groupoid composition on the tangent groupoid, making it into an associative algebroid. We intend to investigate this additional structure in a forthcoming note on multiplicative weightings for Lie groupoids. 
\end{remark}

\subsection{\v{S}evera's algebroid}

Let $V\to M$ be a rank $k$ vector bundle with inner product and assume that the principal ${\rm SO}(k)$-frame bundle admits a lift to a principal ${\rm Spin}(k)$-bundle $P$. Let $\Cl(\R^k)$ be the complexified Clifford algebra, and consider the action of ${\rm Pair}({\rm Spin}(k))$ on ${\rm Pair}(P) \times \Cl(\R^k)$ given by
    \[ (g_1, g_2).(f_1,f_2, v) = (f_1.g_1, f_2.g_2, g_1vg_2^{-1}).  \]

\begin{definition}[cf. \cite{vsevera2017letters}]
    If $V$ is a rank $k$ vector bundle equipped with inner product and spin structure, the \emph{Clifford algebroid} is the associated bundle
        \[ \mathscr{C}l(V) = {\rm Pair}(P)\times_{{\rm Pair}({\rm Spin}(k))}\Cl(\R^k) \to {\rm Pair}(M) \]
\end{definition}

We will show that there is a canonical linear weighting of $\mathscr{C}l(V)$, starting with a technical lemma. The proof of this lemma uses weighted paths so in order to keep our exposition short, we postpone the it to the appendix. 

\begin{lemma}
\label{lemma: action is weighted morphism}
    If ${\rm Pair}(P)$, ${\rm Pair}({\rm Spin}(k))$ are given the doubled trivial weighting along the diagonal, and $\Cl(\R^k)$ is given the linear weighting defined by its filtration by subspaces, then the group action 
        \begin{equation}
        \label{equation: severa action is weighted}
            {\rm Pair}({\rm Spin}(k)) \times {\rm Pair}(P) \times \Cl(\R^k) \to {\rm Pair}(P) \times \Cl(\R^k)
        \end{equation}
    is a weighted morphism.
\end{lemma}

A consequence of this lemma is the following theorem. 

\begin{theorem}
    Let $\pi: {\rm Pair}(P)\times \Cl(\R^k) \to \mathscr{C}l(V)$ be the quotient map. Then 
        \[ C^\infty_{pol}(\mathscr{C}l(V))_{(i)} = \{ f \in C^\infty_{pol}(\mathscr{C}l(V)) : \pi^*f \in C^\infty_{pol}({\rm Pair}(P)\times \Cl(\R^k))_{(i)} \}  \] 
    defines a linear weighting of $\mathscr{C}l(V)$, where ${\rm Pair}(P)\times \Cl(\R^k)$ is given the product weighting. 
\end{theorem}
\begin{proof}
    We have to find weighted vector bundle coordinates. Any point $p\in {\rm Pair}(M)$ is contained in an open neighbourhood $U\sset {\rm Pair}(M)$ such that ${\rm Pair}(P)|_U$ is isomorphic to $U\times {\rm Pair}({\rm Spin}(k))$ as weighted manifolds. Choosing $U$ small enough, we may assume that there exist weighted coordinates $x_a \in C^\infty(U)$. If $y_b$ are linear weighted coordinates on $\Cl(\R^k)$ then since the action map~\eqref{equation: severa action is weighted} is weighted and the following diagram commutes 
        \begin{equation}
        \xymatrix{
            {\rm Pair}(P)|_U \times \Cl(\R^k) \cong U\times {\rm Pair}({\rm Spin}(k)) \times \Cl(\R^k) \ar[dr] \ar[d]_{\pi} & \\
            \mathscr{C}l(V)|_U \ar[r]_\cong & U\times \Cl(\R^k),
        }
        \end{equation}  
    it follows that we can take $x_a, y_b$ to be our weighted vector bundle coordiantes on $\mathscr{C}l(V)|_U$. 
\end{proof}

\begin{definition}
    The \emph{\v{S}evera algebroid} is the weighted deformation bundle
        \[^\tau \mathscr{C}l(V) = \defw(V).  \]
\end{definition}

\appendix

\section{Proof of Lemma~\ref{lemma: action is weighted morphism}}

We now prove Lemma~\ref{lemma: action is weighted morphism}. As indicated above, the proof uses a characterization of weighted morphisms in terms of weighted paths. We being by explaining this characterization and then give a proof of the lemma. The contents of this section are based on communications with Gabriel Beiner, Yiannis Loizides, and Eckhard Meinrenken. 

\subsection{Weighted paths}

Let $(M,N)$ be a weighted pair. A path $\gamma:\R \to M$ is called a \emph{weighted path} if it is a weighted morphism for the trivial weighting of $\R$ along the origin. In particular, if $\gamma$ is a weighted path and $f \in C^\infty(M)_{(i)}$, then 
    \[ \gamma(0) \in N \quad \text{and} \quad f(\gamma(t)) = O(t^i).  \]
If $x_a$ is a system of weighted coordinates, with weight $w_a$, then $\gamma$ is of the form 
    \[ \gamma(t) = (x_1(t), \dots, x_n(t)) = (O(t^{w_1}), \dots, O(t^{w_n})). \]
    
It turns out that the weighting on $M$ can be characterized in terms of weighted paths. 

\begin{proposition}
\label{A-proposition: characterization of weighted morphisms}
    \begin{enumerate}
        \item We have $f \in C^\infty(M)_{(i)}$ if and only if  
            \[ f(\gamma(t)) =  O(t^i)\]
        for every weighted path $\gamma:\R \to M$.

        \item A smooth map $\phi:(M, N) \to (M', N')$ between weighted pairs is a weighted morphism if and only if it takes weighted paths to weighted paths. 
    \end{enumerate}
\end{proposition}
\begin{proof}
    \begin{enumerate}
        \item It remains to prove that a function $f\in C^\infty(M)$ with the property that $f(\gamma(t)) =  O(t^i)$ for every weighted path $\gamma :\R \to M$ has filtration degree $i$. 

    Let $p \in N$ and let $x_a$ be weighted coordinates defined on $U\sset M$. Consider the Taylor expansion of $f$ with respect to these coordinates:
        \[ f(x) = \sum_{I \cdot w < i} c_Ix^I + g(x) = \sum_{j=1}^{i-1}p_j(x) + g(x), \]
    where $I = (i_1, \dots, i_n)$ is a multi-index, $g\in C^\infty(M)_{(i)}$, and $p_j$ are weighted homogeneous of degree $j$. Let $\lambda = (\lambda_1, \dots, \lambda_n) \in \R^n$ and consider the weighted path $\gamma(t) = (\lambda_1t^{w_1}, \dots, \lambda_nt^{w_n})$. Using that the $p_j$ are weighted homogeneous of degree $j$, we have, by assumption, that 
        \[ f(\gamma(t)) = \sum_{j=0}^{i-1}p_j(\gamma(t))+g(\gamma(t)) = \sum_{j=0}^{i-1}t^jp_j(\lambda)+g(\gamma(t)) = O(t^i). \]
    Since $g(\gamma(t)) = O(t^i)$, this implies that $t^jp_j(\lambda) = O(t^i)$, whence $p_j(\lambda) = 0$ for all $\lambda \in \R^n$. Thus, $f = g \in C^\infty(M)_{(i)}$.

    \item Clearly a weighted morphism takes weighted paths to weighted paths, so it remains to prove the converse. 

    Suppose that $\phi$ takes weighted paths to weighted paths. Let $f \in C^\infty(M')_{(i)}$ and let $\gamma : \R \to M$ be a weighted path. Using the previous point, we have that 
        \[ (\phi^*f)(\gamma(t)) = f(\phi(\gamma(t))) = O(t^i), \]
    hence $\phi^*f \in C^\infty(M)_{(i)}$, since $\phi\circ \gamma : \R \to M'$ is a weighted path. \qedhere
    \end{enumerate}
\end{proof}

\subsection{Proof of Lemma~\ref{lemma: action is weighted morphism}}

Let us briefly recall the set up. Let $V \to M$ be a rank $k$ vector bundle with inner product and given spin structure. Let $P\to M$ be the principal ${\rm Spin}(k)$-bundle specified by the spin structure, and recall that the action of ${\rm Pair}({\rm Spin}(k))$ on ${\rm Pair}(P)\times \Cl(\R^k)$ is given by
    \[ (g_1, g_2).(p_1, p_2, v) = (p_1.g_1, p_2.g_2, g_1vg_2^{-1}), \]
where we are making use of the inclusion ${\rm Spin}(k) \sset \Cl(\R^k)$. In order to show that this action is weighted, we make use of weighted paths. 

The weighting on both ${\rm Pair}(P)$ and ${\rm Pair}({\rm Spin}(k))$ is the doubled weighting along the diagonal, where a function has filtration degree $2i$ if it vanishes to order $i$ along the diagonal. Giving ${\rm Pair}(M)$ the doubled weighting along the diagonal, one finds that ${\rm Pair}(P)$ can be identified locally with ${\rm Pair}(M)\times {\rm Pair}({\rm Spin}(k))$ as weighted manifolds. Thus, to show that the action of ${\rm Pair}({\rm Spin}(k))$ on ${\rm Pair}(P)$ is weighted, it is sufficient (in fact, equivalent) to show that the action of ${\rm Pair}({\rm Spin}(k))$ on itself is a weighted morphism. 

\begin{lemma}
    Suppose that $G$ is a Lie group and $H\sset G$ is a closed subgroup. If $G$ is given the doubled weighting along $H$, then the map 
        \begin{align*}
            a: G\times G & \to G \\
            (g_1, g_2) & \mapsto g_1g_2^{-1}
        \end{align*}
    is a weighted morphism. 
\end{lemma}
\begin{proof}
    By definition of the doubled weighting, the map $a$ is weighted if and only if $a^*\van{H} \sset \van{H\times H}$. From this it follows that $a$ is a weighted morphism, since $H$ is a subgroup.  
\end{proof}

It remains to show that the action of ${\rm Pair}({\rm Spin}(k))$ on $\Cl(\R^k)$ is weighted, which we accomplish by making use of weighted paths. It is sufficient to consider path $\gamma$ in ${\rm Pair}({\rm Spin}(k)) \times \Cl(\R^k)$ of the form 
    \[ \gamma(t) = \left(g\exp(\xi t)\exp(\xi_1(t)), g\exp(\xi t)\exp(\xi_2(t)), \sum_{i=0}^k c_it^i\right),  \]
where $g \in {\rm Spin}(k)$, $\xi \in \mathfrak{spin}(k)$, $\xi_1(t), \xi_2(t) = O(t^2)$, and $c_i \in \Cl_i(\R^k)$. Composing this with the action ${\rm Pair}({\rm Spin}(k))\times \Cl(\R^K)\to \Cl(\R^k)$ gives 
    \begin{equation}
    \label{A-equation: weighted path calculation}
        \sum_{i=0}^k {\rm Ad}_{g\exp(\xi t)}\left(\exp(\xi_1(t))c_i\exp(-\eta_2(t))t^i\right).
    \end{equation}
Using that the adjoint action of ${\rm Spin}(k)$ on $\Cl(\R^k)$ is filtration preserving, the identification $\mathfrak{spin}(k) = \Cl_2(\R^k)$, and the fact that $\xi_1(t), \xi_2(t) = O(t^2)$, it follows from the power series expansion of the exponential map that 
    \[ \sum_{i=0}^k {\rm Ad}_{g\exp(\xi t)}\left(\exp(\xi_1(t))c_i\exp(-\eta_2(t))t^i\right) = \sum_{i=0}^k v_i t^i + O(t^{k+1}), \]
for some $v_i \in \Cl_i(\R^k)$. In particular,~\eqref{A-equation: weighted path calculation} is a weighted path. By Proposition~\ref{A-proposition: characterization of weighted morphisms}, this completes the proof of Lemma~\ref{lemma: action is weighted morphism}.

\end{document}